\documentclass[12pt]{article}

\usepackage{listings}
\usepackage{verbatim}
\usepackage{setspace}
\usepackage{sectsty}
\subsectionfont{\normalsize}
\makeatletter
\renewcommand\section{\@startsection{section}{1}{\z@}%
                                  {-2.0ex \@plus -1ex \@minus -.2ex}%
                                  {2.0ex \@plus.2ex}%
                                  {\normalfont\normalsize\bfseries}}
\makeatother
\usepackage{fullpage}
\usepackage{url}
\usepackage{amsmath}
\usepackage{mathdots}
\newcommand{\UF}{\infty}
\newcommand{\bdag}{b^\dagger}
\newcommand{\fI}{{f_A^\I}}
\newcommand{\fII}{{f_A^{\II}}}
\newcommand{\gI}{{g_A^\I}}
\newcommand{\gII}{{g_A^{\II}}}
\newcommand{\CI}{{C^\I_A}}
\newcommand{\CII}{{C^{\II}_A}}
\newcommand{\cI}{{c^\I_A}}
\newcommand{\cII}{{c^{\II}_A}}
\newcommand{\lI}{{\ell^\I}}
\newcommand{\lII}{{\ell^{\II}}}

\newcommand{\finsubset}{\subseteq^{\rm fin}}
\newcommand{\NIM}{{\rm NIM}}
\newcommand{\WAcash}{W_A^{\rm cash}}
\newcommand{\I}{{\rm I } }
\newcommand{\II}{{\rm II  }}

\newcommand{\IIns}{{\rm II}}

\newcommand{\ang}[1]{\langle#1\rangle}

\newcommand{\nat}{{\sf N}}

\newcommand{\xvec}[1]{\ifcase 3{#1} {\ang {x_1,x_2,x_3} } \else 
\ifcase 4{#1} {\ang{x_1,x_2,x_3,x_4}} \else {\ang {x_1,\ldots,x_{#1}}}\fi\fi}
\newcommand{\yvec}[1]{\ifcase 3{#1} {\ang {y_1,y_2,y_3} } \else 
\ifcase 4{#1} {\ang{y_1,y_2,y_3,y_4}} \else {\ang {y_1,\ldots,y_{#1}}}\fi\fi}
\newcommand{\zvec}[1]{\ifcase 3{#1} {\ang {z_1,z_2,z_3} } \else 
\ifcase 4{#1} {\ang{z_1,z_2,z_3,z_4}} \else {\ang {z_1,\ldots,z_{#1}}}\fi\fi}
\newcommand{\vecc}[2]{\ifcase 3{#2} {\ang { {#1}_1,{#1}_2,{#1}_3 } } \else
\ifcase 4{#1} {\ang { {#1}_1,{#1}_2,{#1}_3,{#1}_{4} } }
\else {\ang { {#1}_1,\ldots,{#1}_{#2}}}\fi\fi}
\newcommand{\veccd}[3]{\ifcase 3{#2} {\ang { {#1}_{{#3}1},{#1}_{{#3}2},{#1}_{{#3}3} } } \else
\ifcase 4{#1} {\ang { {#1}_{{#3}1},{#1}_{{#3}2},{#1}_{#3}3},{#1}_{{#3}4} }
\else {\ang { {#1}_{{#3}1},\ldots,{#1}_{{#3}{#2}}}}\fi\fi}
\newcommand{\veccz}[2]{\ifcase 3{#2} {\ang { {#1}_0,{#1}_2,{#1}_3 } } \else
\ifcase 4{#1} {\ang { {#1}_0,{#1}_2,{#1}_3,{#1}_{4} } }
\else {\ang { {#1}_0,\ldots,{#1}_{#2}}}\fi\fi}
\newcommand{\xve}[1]{\ifcase 3{#1} {x_1,x_2,x_3} \else 
\ifcase 4{#1} {x_1,x_2,x_3,x_4} \else {x_1,\ldots,x_{#1}}\fi\fi}
\newcommand{\yve}[1]{\ifcase 3{#1} {y_1,y_2,y_3} \else 
\ifcase 4{#1} {y_1,y_2,y_3,y_4} \else {y_1,\ldots,y_{#1}}\fi\fi}
\newcommand{\zve}[1]{\ifcase 3{#1} {z_1,z_2,z_3} \else 
\ifcase 4{#1} {z_1,z_2,z_3,z_4} \else {z_1,\ldots,z_{#1}}\fi\fi}
\newcommand{\ve}[2]{\ifcase 3#2 {{#1}_1,{#1}_2,{#1}_3} \else
\ifcase 4#2 {{#1}_1,{#1}_2,{#1}_3,{#1}_{4}}
\else {{#1}_1,\ldots,{#1}_{#2}}\fi\fi}
\newcommand{\ved}[3]{\ifcase 3#2 {{#1}_{{#3}1},{#1}_{{#3}2},{#1}_{{#3}3}} \else
\ifcase 4#2 {{#1}_{{#3}1},{#1}_{{#3}2},{#1}_{{#3}3},{#1}_{{#3}4}}
\else {{#1}_{{#3}1},\ldots,{#1}_{{#3}{#2}}}\fi\fi}
\newcommand{\fuve}[3]{
\ifcase 3#2
{{#3}({#1}_1),{#3}({#1}_2,{#3}({#1}_3)} \else
\ifcase 4#2
{{#3}({#1}_1),{#3}({#1}_2),{#3}({#1}_3),{#3}({#1}_4)}
\else
{{#3}({#1}_1),\ldots,{#3}({#1}_{#2})}\fi\fi}
\newcommand{\setmathchar}[1]{\ifmmode#1\else$#1$\fi}
\newcommand{\vlist}[2]{%
	\setmathchar{%
		\compound#2\one{#2}\two
		\ifcompound
			({#1}_1,\ldots,{#1}_{#2})
		\else
			\ifcat N#2
				({#1}_1,\ldots,{#1}_{#2})
			\else
				\ifcase#2
					({#1}_0)\or
					({#1}_1)\or
					({#1}_1,{#1}_2)\or 
					({#1}_1,{#1}_2,{#1}_3)\or
					({#1}_1,{#1}_2,{#1}_3,{#1}_4)\else 
					({#1}_1,\ldots,{#1}_{#2})
				\fi
			\fi
		\fi}}

\newif\ifcompound
\def\compound#1\one#2\two{%
	\def\one{#1}
	\def\two{#2}
	\if\one\two
		\compoundfalse
	\else
		\compoundtrue
	\fi}

\newcommand{\xwe}[1]{\ifcase 3{#1} {x_1\wedge x_2\wedge x_3} \else 
\ifcase 4{#1} {x_1\wedge x_2\wedge x_3\wedge x_4} \else {x_1\wedge \cdots \wedge
x_{#1}}\fi\fi}
\newcommand{\we}[2]{\ifcase 3#2 {\ang { {#1}_1\wedge {#1}_2\wedge {#1}_3 } } \else
\ifcase 4{#1} {\ang { {#1}_1\wedge {#1}_2\wedge {#1}_3\wedge {#1}_{4} } }
\else {\ang { {#1}_1\wedge \cdots\wedge {#1}_{#2}}}\fi\fi}

\newcommand{\st}{\mathrel{:}}

\newcommand{\ceil}[1]{\left\lceil {#1}\right\rceil}
\newcommand{\floor}[1]{\left\lfloor{#1}\right\rfloor}

\newcommand{\s}[1]{\s_{#1}}

\newcommand{\monus}{\;\raise.5ex\hbox{{${\buildrel
    \ldotp\over{\hbox to 6pt{\hrulefill}}}$}}\;}

\newcounter{savenumi}

\newtheorem{theoremfoo}{Theorem}[section] 
\newenvironment{theorem}{\pagebreak[1]\begin{theoremfoo}}{\end{theoremfoo}}

\newtheorem{lemmafoo}[theoremfoo]{Lemma}
\newenvironment{lemma}{\pagebreak[1]\begin{lemmafoo}}{\end{lemmafoo}}
\newtheorem{conjecturefoo}[theoremfoo]{Conjecture}

\newenvironment{conjecture}{\pagebreak[1]\begin{conjecturefoo}}{\end{conjecturefoo}}

\newtheorem{conventionfoo}[theoremfoo]{Convention}

\newtheorem{porismfoo}[theoremfoo]{Porism}

\newtheorem{gamefoo}[theoremfoo]{Game}

\newtheorem{corollaryfoo}[theoremfoo]{Corollary}

\newtheorem{openfoo}[theoremfoo]{Open Problem}

\newtheorem{exercisefoo}{Exercise}

\newcommand{\fig}[1] 
{
 \begin{figure}
 \begin{center}
 \input{#1}
 \end{center}
 \end{figure}
}

\newtheorem{potanafoo}[theoremfoo]{Potential Analogue}

\newtheorem{notefoo}[theoremfoo]{Note}
\newenvironment{note}{\pagebreak[1]\begin{notefoo}\rm}{\end{notefoo}}

\newtheorem{notabenefoo}[theoremfoo]{Nota Bene}

\newtheorem{nttn}[theoremfoo]{Notation}
\newenvironment{notation}{\pagebreak[1]\begin{nttn}\rm}{\end{nttn}}

\newtheorem{empttn}[theoremfoo]{Empirical Note}

\newtheorem{examfoo}[theoremfoo]{Example}
\newenvironment{example}{\pagebreak[1]\begin{examfoo}\rm}{\end{examfoo}}

\newtheorem{dfntn}[theoremfoo]{Def}
\newenvironment{definition}{\pagebreak[1]\begin{dfntn}\rm}{\end{dfntn}}

\newtheorem{propositionfoo}[theoremfoo]{Proposition}

\newenvironment{proof}
    {\pagebreak[1]{\narrower\noindent {\bf Proof:\quad\nopagebreak}}}{\QED}

\newcommand{\yyskip}{\penalty-50\vskip 5pt plus 3pt minus 2pt}
\newcommand{\blackslug}{\hbox{\hskip 1pt
        \vrule width 4pt height 8pt depth 1.5pt\hskip 1pt}}
\newcommand{\QED}{{\penalty10000\parindent 0pt\penalty10000
        \hskip 8 pt\nolinebreak\blackslug\hfill\lower 8.5pt\null}
        \par\yyskip\pagebreak[1]}

\newcommand{\BBB}{{\penalty10000\parindent 0pt\penalty10000
        \hskip 8 pt\nolinebreak\hbox{\ }\hfill\lower 8.5pt\null}
        \par\yyskip\pagebreak[1]}

\newtheorem{factfoo}[theoremfoo]{Fact}

\newenvironment{block}{\begin{list}{\hbox{}}{\leftmargin 1em
    \itemindent -1em \topsep 0pt \itemsep 0pt \partopsep 0pt}}{\end{list}}

\everymath{\displaystyle}

\begin{document}

\title{NIM with Cash}

\author{
{William Gasarch}
\thanks{University of Maryland,
Dept. of Computer,
        College Park, MD\ \ 20742.
\texttt{gasarch@cs.umd.edu}
}
\\ {\small Univ. of MD at College Park}
\and
{John Purtilo}
\thanks{University of Maryland,
Dept. of Computer Science,
        College Park, MD\ \ 20742.
\texttt{jep1911@umd.edu}
}
\\ {\small Univ. of MD at College Park}
\and
{Douglas Ulrich}
\thanks{University of Maryland,
Dept. of Mathematics,
        College Park, MD\ \ 20742.
\texttt{ds\_ulrich@hotmail.com}
}
\\ {\small Univ. of MD at College Park}
}

\date{}

\maketitle

\begin{abstract}
Let $A$ be a finite subset of $\nat$.
Then $\NIM(A;n)$ is the following 2-player game:
initially there are $n$ stones on the board
and the players alternate removing $a\in A$
stones. The first player who cannot move loses.
This game has been well studied. 

We investigate an extension of the game where Player \I starts out
with $d$ dollars, Player \II starts out with $e$ dollars,
and when a player removes $a\in A$ he loses $a$ dollars.
The first player who cannot move loses; however, note this can happen 
for two different reasons: (1) the number of stones is less than $\min(A)$,
(2) the player has less than $\min(A)$ dollars.
This game leads to more complex win conditions then standard NIM.

We prove some general theorems from which we can obtain win conditions
for a large variety of finite sets $A$. We then apply them to
the sets $A=\{1,L\}$, and $A=\{1,L,L+1\}$.
\end{abstract}

\section{Introduction}

\begin{notation}
$A\finsubset B$ means that $A$ is a finite subset of $B$.
\end{notation}

\begin{definition}
Let $A\finsubset\nat$ and let $n\in\nat$.
Let $a_1=\min(A)$.
$\NIM(A;n)$ is played as follows:
\begin{enumerate}
\item 
There are two players, Player I and Player II.
They alternate moves with Player I going first.
During a player's turn he removes $a\in A$ stones from the board.
\item
Initially there are $n$ stones on the board.
\item
If a player cannot move he loses. 
If there are $s$ stones on the board and $s<a_1$ then the player loses.
\end{enumerate}
\end{definition}

\begin{notation}
We will usually omit the word {\it stones} and just
say that a Player {\it removes X } rather than
{\it removes X stones}.
\end{notation}

\begin{notation}~
The expression {\it Player I wins} means that Player I has a strategy that
will win regardless of what Player II does.
If $A\finsubset\nat$ and $n\in\nat$ then we write $W_A(n)=\I$ if Player \I wins.
Similar for {\it Player II wins}.
If the game is understood we may simply use $W$.
\end{notation}

NIM is an example of a {\it combinatorial game}. Such games
have a vast literature 
(see the selected bibliography of Frankel~\cite{combgamesbib}).
Variants on the 1-pile version have included
letting the number of stones a player can remove
depend on how many stones are in the pile~\cite{NIMarb},
letting the number of stones a player can remove depend on the player~\cite{partizan},
allowing three players~\cite{threenim},
viewing the stones as cookies that may spoil~\cite{cookie},
and others. 
Grundy~\cite{grundy} and Sprague~\cite{sprag} showed how to analyze
many-pile NIM games by analyzing the 1-pile NIM games that it consists of.
NIM games are appealing because they are easy to explain,
yet involve interesting (and sometimes difficult) mathematics
to analyze.

We give several examples of known win-loss patterns for NIM-games.

\begin{example}\label{ex:nimexamples}~
We specify a NIM game by specifying the $A\finsubset\nat$.
\begin{enumerate}
\item
Let $1\le L < M$,
$A=\{L,\ldots,M\}$.
$W(n)=\II$ iff
$n\equiv 0,1,2,\ldots,L-1 \pmod {L+M}$.

\item
Let $L\ge 2$, even and $A=\{1,L\}$.
$W(n)=\II$ iff 
$n\equiv 0,2,4,\ldots,L-2 \pmod {L+1}$.
(We leave it to the reader to show that the case of $L$ odd is boring.)

\item
Let $L\ge 2$, even and
$A=\{1,L,L+1\}$.
$W(n)=\II$ iff $n\equiv 0,2,4\ldots,L-2 \pmod {2L}$.

\item
Let $L\ge 3$, odd and
$A=\{1,L,L+1\}$.
$W(n)=\II$ iff $n\equiv 0,2,4,\ldots,L-1 \pmod {2L+1}$.

\end{enumerate}
\end{example}

This paper is about the following variant of NIM which we refer to as {\it NIM with Cash}.

\begin{definition}
Let $A\finsubset\nat$ and let $n;d,e\ge 0$.
Let $a_1=\min\{A\}$.  $\NIM(A;n;d,e)$ is played as follows:
\begin{enumerate}
\item 
There are two players, Player \I and Player \IIns.
They alternate moves with Player I going first.
During a player's turn he removes $a\in A$ stones from the board and loses $a$ dollars.
\item
Initially there are $n$ stones on the board, Player I has $d$ dollars,
Player II has $e$ dollars.
\item
If a player cannot move he loses.  
If there are $s$ stones and the player has $f$ dollars and either $s<a_1$ or $f<a_1$ then the player loses.
\end{enumerate}
\end{definition}

\begin{notation}~
The expression {\it Player \I wins} means that Player I has a strategy that
will win regardless of what Player II does.
If the game is using the set $A\finsubset\nat$ 
then we write $\WAcash(n;d,e)=\I$ if Player \I  wins when
the board initially has $n$ stones, 
Player I has $d$ dollars, and Player II  has $e$ dollars.
Similar for {\it Player \II wins}.
If the game is understood we may simply use $W$.  
\end{notation}

\begin{notation}
If $A$ is a set then $\NIM(A)$ will mean the NIM-with-Cash game
with the set $A$.
\end{notation}

\begin{definition}
Let $A\finsubset\nat$.  Assume we are playing $\NIM(A)$.
Let $n,d,e\ge 0$.
We define the state of the game. In all cases it is Player I's turn.
\begin{enumerate}
\item
The game is in {\it state $(n;d,e)$}
if there are $n$ stones on the board, Player I has $d$ dollars,
and Player II has $e$ dollars.
\item
The game is in {\it state $(n;d,\UF)$}
if there are $n$ stones on the board, Player I has $d$ dollars and
Player II has unlimited funds.
State $(n;\UF,e)$ is defined similarly.
Note that $(n;\UF,\UF)$ is the standard NIM game.
\end{enumerate}
\end{definition}

\begin{definition}
Let $A\finsubset\nat$. We are concerned with $\NIM(A)$.
\begin{enumerate}
\item
If $W(n)=\I$  and Player I has enough money to 
play the strategy he would play to win in standard NIM
then we say that Player I  {\it wins normally}. 
Note that this is the same as saying $W(n;d, \infty) = \I$.
Similar for $W(n)=\II$.
\item
If either player wins by removing $\min(A)$ on every turn then we say 
he wins {\it miserly}.
The intuition is that he wins because the other 
player ran of money though formally this might
not be the case.
\end{enumerate}
\end{definition}

We give several examples of play.

\begin{example}
In all of the examples below $A=\{1,3,4\}$ and $n=14$.

\begin{enumerate}
\item
The game is in state $(14;\UF,10)$.
In standard NIM Player II would win $W(14)$  by always making sure that Player I
faces an $n\equiv 0,2 \pmod 7$.
By a case analysis one can show that 
Player II has enough money to play this strategy.
Hence Player II wins normally.
\item
The game is in state $(14;4,4)$. The reader can check that if Player II always removes one then
he will win miserly.
\item
The game is in state $(14;9,9)$.
We show that Player I wins using a strategy that begins  miserly but may
becomes normal.
Player I removes 1. Player II removes $a\in \{1,3,4\}$.

\begin{enumerate}
\item
If $a=4$ then the state is $(9;8,5)$.
The reader can verify that Player I wins miserly.

\item
If $a=3$ then the state is $(10;8,6)$. 
The reader can verify that Player I wins miserly or normally.

\item
If $a=1$ then the state is $(12;8,8)$.
Player I removes 1. Player II removes $a\in \{1,3,4\}$.
\begin{enumerate}
\item
If $a=4$ then the state is $(7;7,4)$. 
The reader can verify that Player I wins miserly.
\item
If $a=3$ then the state is $(8;7,5)$. 
The reader can verify that Player I wins miserly or normally.
\item
If $a=1$ then the state is $(10;7,7)$.
The reader can verify that Player I wins normally.
\end{enumerate}
\end{enumerate}
\end{enumerate}
\end{example}

The following lemma and definition will be useful throughout the entire paper.
The lemma is so ubiquitous that we will use it without mention.

\begin{lemma}\label{le:helpful}
Let $A\finsubset\nat$.
Let $d,e\in \nat$.
Assume that the game is $\NIM(A;d,e)$.
$$W(n;d,e)=\I \iff (\exists a\in A, a\le n,d)[W(n-a;e,d-a)=\IIns].$$
\end{lemma}

\begin{definition}
Assume $W(n;d,e)=\I$.  
Let $a\in A$. If  $W(n-a;e,d-a)=\II$ then we call $a$ {\it a winning move}.
 If  $W(n-a;e,d-a)=\I$ then we call $a$ {\it a losing move}.
\end{definition}

We are interested in the following problem: Given $A$ find a win condition
for $\NIM(A)$. One could write a dynamic program that, on input
$(n;d,e)$, determines who wins in $O(n^3)$ arithmetic operations, but we want
our win conditions to be simpler than that.
\begin{definition}
Let $A$ be a finite set. A {\it win condition for $\NIM(A)$ } is a
polynomial time function of the length of $(n;d,e)$. Since $n,d,e$ are in binary
we want a polynomial time function of $O(\log(nde))$.
\end{definition}

In Section~\ref{se:rich} we define "rich" and explore the case where at least one player is rich.
In Section~\ref{se:poor} we define "poor" and explore the case where at least one player is  poor.
In Section~\ref{se:middleclass} we explore the case where neither player is rich or poor.
The theorems proven allow one to obtain nice win conditions for many sets $A$.
In Sections~\ref{se:1L}, \ref{se:1LLodd}, \ref{se:1LLeven} we obtain win conditions for
$A=\{1,L\}$, $A=\{1,L,L+1\}$ ($L$ odd), and 
$A=\{1,L,L+1\}$ ($L$ even).
In Section~\ref{se:LM} we state a conjecture about the set $A=\{L,\ldots,M\}$.
In Section~\ref{se:questions} we suggest  future directions.

\section{What if At Least One Player Is Rich?}\label{se:rich}

If $W_A(n)=\I$ then how much money does Player \I need to win normally
starting with $n$ stones? A similar question could be asked about Player \IIns.
In this section we define $\fI(n)$ and $\fII(n)$ to be those amounts.
We then consider what happens if (say) Player \I has $\fI(n)-1$ dollars.
How much does Player \II need to snatch victory from the jaws of defeat?

For $n$ such that $W_A(n)=\I$ we define $\fI(n)$.
Later we will see that this $\fI(n)$ is the least $d$ such that
Player \I wins $(n;d,\UF)$. Similarly for $W_A(n) = \II$.

\begin{definition}\label{de:ff}
Let $A\finsubset \nat$ and let $a_1=\min(A)$.
\begin{enumerate}
\item
For $0\le n\le a_1-1$
$\fII(0)=0$ (Player \II wins and needs 0 to win.)
\item
If $W_A(n)=\I$ then

$$\fI(n)= \min_{a\in A, a\le n}\{ \fII(n-a)+a \st  W_A(n-a)=\II \}$$

\item

If $W_A(n)=\II$ then

$$\fII(n)= \max_{a\in A, a\le n}\{ \fI(n-a)\}$$
\end{enumerate}
\end{definition}

The following is a straightforward proof by induction.

\begin{theorem}\label{th:rich}
Let $A\finsubset \nat$ and $\fI$, $\fII$ be as defined above.
\begin{enumerate}
\item
If $W_A(n)=\I$ then Player \I wins $(n;\fI(n),\UF)$.
\item
If $W_A(n)=\II$ then Player \II wins $(n;\UF,\fII(n))$.
\end{enumerate}
\end{theorem}

Note that $\fI$ 
($\fII$) is only defined when on $n$ such that $W_A(n)=\I$
($W_A(n)=\II$) and we do not know what happens if (say) $W_A(n)=\I$ but
Player \I has $\fI(n)-1$. We now complete the definitions of $\fI$ and $\fII$ to
deal with these questions. 

\begin{definition}\label{de:f}
Let $A\finsubset \nat$. Let $a_1=\min(A)$.
\begin{enumerate}
\item
If $W_A(n)=\I$ ($W_A(n)=\II$) then $\fI(n)$ ($\fII(n)$) is defined as in Definition~\ref{de:ff}.
\item
If $0\le n\le a_1-1$ then $\fI(n)=0$.
\item
If $W_A(n) = \II$ then 

\[\fI(n) = \min_{a\in A, a\le n}\{\fII(n-a) + a: \fI(n-a) = \fII(n)\}.\]

Since $\fII(n)= \max_{a\in A, a\le n}\{ \fI(n-a)\}$ we know that the set of $a\in A$,
$q\le n$ such that $\fI(n-a)=\fII(n)$ is not empty.

\item
It $W_A(n) = \I$ then

\[\fII(n) = \max_{a\in A, a\le n} \{\fI(n-a)\}.\]

\end{enumerate}
\end{definition}

The following theorem has an easy proof that uses
Theorem~\ref{th:rich} and a straightforward induction.

\begin{theorem}\label{th:wrong}
Let $A, n, d, e$ be given. Let $\fI,\fII$ be as defined above.
\begin{enumerate}
\item If $d \geq \fI(n)$ and $e < \fII(n)$ then $\WAcash(n;d,e) = \I$.
\item If $d < \fI(n)$ and $e \geq \fII(n)$ then $\WAcash(n;d,e) = \II$
\item If $d \geq \fI(n)$ and $e \geq \fII(n)$ then $\WAcash(n;d,e) = W_A(n)$.
(This follows from Theorem~\ref{th:rich}. We include it so that we can just
refer to this theorem for all cases where at least one player is rich.)
\end{enumerate}
\end{theorem}

\section{What if At Least One Player Is Poor?}\label{se:poor}

Let $A\finsubset\nat$ and $n,d,e\in\nat$ be such that $d$ and $e$ are small--
so small that the best strategy is to play miserly. If $d\le e$ then Player II will win since
Player I will run out of money first. If $d>e$ then it is not clear what happens.

We will define formulas and state a theorem that will make this all rigorous.
In the end we will have determined which player wins if at least one player is 
poor.

\begin{definition}\label{de:g}
Let $A$, $n$ be given. Set $a_1 = \min(A)$. Let $i\equiv n \mod {2a_1}$. Then:
\begin{itemize}
\item $\gI(n) = \frac{n-i}{2} + \min\{i+1,a_1\}$.
\item $\gII(n) = \frac{n-i}{2} + \max\{0, i-a_1 + 1\}$.
\end{itemize}
\end{definition}

From the definitions of $\gI(n)$ and $\gII(n)$ one can easily prove the following.

\begin{lemma}\label{le:first}
For all $n, k$, we have $\gI(n+2ka_1) = \gI(n) + ka_1$ and $\gII(n+2ka_1) = \gII(n)+ka_1$.
\end{lemma}

\begin{lemma}\label{le:d}
Let $a\in A$.  
If $d<\gI(n)$ then $d-a<\gII(n-a)$.
\end{lemma}

\begin{proof}
We show that $\gI(n) - a \leq \gII(n-a)$.

\noindent
{\bf Claim 1:} If $a\in A$ and $a<2a_1$ then $\gI(n) -a \leq \gII(n-a)$.

\noindent
{\bf Proof of Claim 1:}

Let $a = a_1 + j$ where $0 \leq j < a_1$. 
Let $i \equiv n \pmod{2a_1}$ and let $i' \equiv n-a \pmod{2a_1} = (i-a_1-j) \pmod{2a_1}$.

We need: $\frac{n-i}{2} + \min\{i+1, a_1\} - a \leq \frac{n-a-i'}{2} + \max\{0, i'-a_1 + 1\}$, i.e. \\ $\frac{n-i}{2} + \min\{i+1, a_1\} \leq \frac{n-i'}{2} + \max\{0, i'-a_1 + 1\} + \frac{a_1 + j}{2}$.

There are two cases, depending on whether $i \geq a_1 + j$ or $i < a_1 + j$. We leave the easy algebra to the reader.

\noindent
{\bf End of Proof of Claim 1}

\noindent {\bf Claim 2:} If $a \in A$ and $2a_1 \leq a < 2a_1 + a_1$ then $\gI(n) - a \leq \gII(n-a)$.

\noindent {\bf Proof of Claim 2:} 

Let $a = 2a_1 + j$ where $0 \leq j < a_1$. 
Let $i \equiv n \pmod{2a_1}$ and let $i' \equiv n-a \pmod{2a_1} = (i-j) \pmod{2a_1}$.

We need: $\frac{n-i}{2} + \min\{i+1, a_1\} - a \leq \frac{n-a-i'}{2} + \max\{0, i'-a_1 + 1\}$, i.e. \\ $\frac{n-i}{2} + \min\{i+1, a_1\} \leq \frac{n-i'}{2} + \max\{0, i'-a_1 + 1\} + a_1 + \frac{ j}{2}$.

There are two cases, depending on whether $i \geq  j$ or $i < j$. We leave the easy algebra to the reader.

\noindent {\bf End of proof of Claim 2}

We now prove the theorem for $a\ge 2a_1$.
Let $a-a_1=2ka_1+i$ where $0\le i\le 2a_1=1$.

\noindent
$\gII(n-a) = \gII(n- a_1 -i -2ka_1) = \gII(n-a_1-i) - ka_1 \hbox{ by Lemma~\ref{le:first} }.$

\noindent
$\gII(n-(a_1+i)) - ka_1 \geq \gII(n-a_1-i) - 2ka_1 \geq \gI(n) - i-a_1 - 2ka_1 = \gI(n) - a \hbox{ by Claims 1 and 2 }.$

Hence $\gII(n-a) \ge \gI(n)-a$.
\end{proof}

\begin{lemma}\label{le:e}
If $e<\gII(n)$ then $e<\gI(n-a_1)$.
\end{lemma}

\begin{proof}

We need to prove that $\gII(n)\le \gI(n-a_1)$. But in fact they are equal.

Let $i = n \pmod{2a_1}$ and let $i' = n-a_1 \pmod{2 a_1}$.

We need to show $\frac{n-i}{2} + \max\{0, i-a_1 + 1\} = \frac{n-a_1-i'}{2} + \min\{i'+1, a_1\}$.

There are two cases depending on whether $0\le i \le  a_1-1$ or $a_1\le i\le 2a_1-1$.
We leave the easy algebra to the reader.

\end{proof}

\begin{theorem}\label{th:poor1}
Let $A, n, d, e$ be given. Then:
\begin{enumerate}
\item If $e < \gII(n)$ and $\floor{\frac{d}{a_1}} > \floor{\frac{e}{a_1}}$ 
then $\WAcash(n;d,e)=\I$.
\item If $d < \gI(n)$ and $\floor{\frac{d}{a_1}} \le \floor{\frac{e}{a_1}}$ 
then $\WAcash(n;d,e) = \II$.
\end{enumerate}
We refer to the above statements as Parts.
\end{theorem}

\begin{proof}
We prove this by induction on $n$.

\noindent
{\bf Base Case:} Assume $0\le n\le a_1-1$.
Note that $n=i$ so $\gI(n)=\min(n+1,a_1)=n+1$, 
$\gII(n)=\max(0,n-a_1+1)=0$.
Part 1 cannot occur since its premise is $e<0$.
Part 2 has as a premise $d<\gI(n)=n+1\le a_1$, 
hence Player \I cannot move so he loses.

\noindent
{\bf Induction Step:} Assume $n\ge a_1$ and that for all $n'<n$ the lemma holds.

\noindent
{\bf Part 1)} 
We show that Player \I removing $a_1$ is a winning move. We by show $\WAcash(n-a_1;e,d-a_1)=\II$ by inducting into Part 2.
By Lemma~\ref{le:e} $e<\gI(n-a_1)$.
From $\floor{\frac{d}{a_1}} > \floor{\frac{e}{a_1}}$ 
one can deduce
$\floor{\frac{d-a_1}{a_1}} \ge \floor{\frac{e}{a_1}}$.

\noindent
{\bf Part 2)} We show $(\forall a\in A)[\WAcash(n-a;e,d-a)=\I]$ 
by inducting into Part 1.
By Lemma~\ref{le:d} $d-a<\gI(n-a)$.
From $\floor{\frac{d}{a_1}} \le \floor{\frac{e}{a_1}}$ 
one can deduce
$\floor{\frac{e}{a_1}} > \floor{\frac{d-a_1}{a_1}}$.
\end{proof}

We prove a lemma which will show that Theorem~\ref{th:poor1} covered
all the cases. We will then state a clean Theorem where all the cases
are clearly spelled out.

\begin{lemma}\label{le:helppoor}
Let $A, n, d, e$ be given. Then:
\begin{enumerate}
\item
If $e<\gII(n)$ and $d\ge \gI(n)$ 
then $\floor{\frac{d}{a_1}} > \floor{\frac{e}{a_1}}$.
\item
If $d<\gI(n)$ and $e\ge \gII(n)$ 
then $\floor{\frac{d}{a_1}} \le \floor{\frac{e}{a_1}}$.
\end{enumerate}
\end{lemma}

\begin{proof}
This can be proved by taking $n=2a_1n'+i$ with $0\le i\le 2a_1-1$
and breaking into the cases $0\le i\le a_1-1$ and $a_1\le i\le 2a_1-1$.
\end{proof}

\begin{theorem}\label{th:poor}
Let $A, n, d, e$ be given. Then:
\begin{enumerate}
\item
If $d\ge \gI(n)$ and $e<\gII(n)$ then $\WAcash(n;d,e)=\I$.
\item
If $d<\gI(n)$ and $e\ge \gII(n)$ then $\WAcash(n;d,e)=\II$.
\item
If $d<\gI(n)$ and $e<\gII(n)$ then $\WAcash(n;d,e)=\I$ iff
$\floor{\frac{d}{a_1}} > \floor{\frac{e}{a_1}}$.
\end{enumerate}
\end{theorem}

\begin{proof}
We {\it do not} use induction. We need only use Theorem~\ref{th:poor1} and Lemma~\ref{le:helppoor}.

\noindent
a) If $d\ge \gI(n)$ and $e<\gII(n)$ then, by Lemma~\ref{le:helppoor},
$\floor{\frac{d}{a_1}} > \floor{\frac{e}{a_1}}$. 
By Theorem~\ref{th:poor1}, $\WAcash(n;d,e)=\I$.

\noindent
b) If $d<\gI(n)$ and $e\ge \gII(n)$ then, by Lemma~\ref{le:helppoor},
$\floor{\frac{d}{a_1}} \le \floor{\frac{e}{a_1}}$. 
By Theorem~\ref{th:poor1}, $\WAcash(n;d,e)=\II$.

\noindent
c) This follows directly from Theorem~\ref{th:poor1}.
\end{proof}

\section{What if Both Players are Middle Class?}\label{se:middleclass}

Let $A\finsubset \nat$ be given and understood for the rest of the paper. 

Theorems~\ref{th:rich} and \ref{th:poor} cover the cases where at least one player is rich
or at least one player is poor. We now deal with the remaining cases.

\begin{definition}
Let $A\finsubset \nat$ and let $\fI, \fII, \gI, \gII$ 
be as in definitions~\ref{de:f} and \ref{de:g}.
If $(n;d,e)$ satisfies 
$\gI(n) \leq d < \fI(n)$ and $\gII(n) \leq e < \gI(n)$, 
then call $(n;d,e)$ $A$-critical.
\end{definition}

\begin{note}
One can easily show that, for all $n$, $\gI(n)\le \fI(n)$ and $\gII(n)\le \fII(n)$.
We do not need this result; however, it is a good sanity check on our definitions.
\end{note}

In general, determining $\WAcash(n;d,e)$ when $(n;d,e)$ is $A$-critical 
seems difficult (although for specific $A$ a pattern is usually obvious).

In this section we describe some conditions on $A$ 
that allow us to give a complete winning condition for $A$. 
In the next sections we apply these to particular examples.

Let $(n;d,e)$ be $A$-critical. It turns out that $(n;d,e)$ are not the right
parameters to work with.

\begin{definition}
Let $(n;d,e)$ be the Nim With Cash state (NwCS for short). 
Let $m\in \nat$ (think of it as the periodicity of the ordinary NIM game with $A$).
Let $i_n=n\pmod m$, $b_{n;d,e}=\fI(n)-1-d$, and $\bdag=\fII(n)-1-e$.
Then $(i;b_{n;d,e},\bdag_{n;d,e})$ is the {\it Corresponding State} (CS for short).
Formally we should call it $CS_m$ but in applications $m$ will be understood.
Note that when $(n;d,e)$ is $A$-critical, $b,\bdag \ge 0$.
\end{definition}

When we pass from the NwCS to the CS we lose information.
But this might not be information we need. Imagine the following:
the game is in CS $(i;b,\bdag)$ and the player removes $a\in A$.
We would like to be able to derive the new CS without knowing the NwCS. This motivates the following definitions.

\begin{definition}~
\begin{enumerate}
\item
For all $n\in\nat$ and for all $a\in A$ with $a \leq n$, $\CI(n,a) = \fI(n) - \fII(n-a) + a$.
\item
For all $n\in\nat$ and for all $a\in A$ with $a \leq n$, $\CII(n, a) = \fII(n) - \fI(n-a)$.
\item
$A$ is {\it cash-periodic with period $m$}  
if for all $n_1, n_2$ with $n_1 \equiv n_2 \bmod m$, we have
\begin{itemize}
\item $W_A(n_1) = W_A(n_2)$.
\item For all $a \in A$ with $a \leq n_1, n_2$, $\CI(n_1, a) = \CI(n_2, a)$.
\item For all $a \in A$ with $a \leq n_1, n_2$, $\CII(n_1,a) = \CII(n_2, a)$.
\end{itemize}
$A$ is {\it cash-periodic} if there exists an $m$ such that $A$ is $m$-cash periodic.
We will always take the least such $m$.
\item
Assume $A$ is cash-periodic with period $m$.
Then we define $\cI(i, a) = \CI(n, a)$ 
for some $n \geq a$ with $n \equiv i \bmod m$.
By the definition of $\CI(n,a)$ it does matter which $n\geq a$ we take.
Similarly for $\cII(i, a)$.
\end{enumerate}
\end{definition}

\noindent \textbf{Remark.} 
In the basic NIM case, we could prove every set is periodic (possibly with an offset) using the pigeonhole principle. 
Here that is not possible since  we cannot bound $\CI(n,a)$ and $\CII(n,a)$. 
Indeed, $\{3, 5, 6, 10, 11\}$ is not cash-periodic (this is not obvious), 
even if we were to modify the definition to allow an offset.

The following lemma follows from the definitions.

\begin{lemma}
Suppose $A$ is cash-periodic, and let $(n;d,e)$ be given; suppose $a \in A$ is such that $n \geq a$ and $d \geq a$. Then:

\begin{itemize}\label{le:periodic}
\item $i_{n-a} = i_{n} - a \bmod m$.
\item $\bdag_{n-a;e,d-a} = b_{n;d,e} - \cI(i_n, a)$.
\item $b_{n-a;e,d-a} = \bdag_{n;d,e} - \cII(i_n, a)$.

\end{itemize}
\end{lemma}

Hence if $(n; d, e)$ is any position, then we can determine the CS for $(n-a;e,d-a)$ 
from the CS for $(n;d,e)$, which is what we wanted.

If $A$ is cash-periodic with period $m$ 
then let $\mathbf{CS}_A$ denote the set of all triples $(i;b,\bdag)$, 
where $0 \leq i < m$ and $b, \bdag \geq 0$. 
If a NwCS is $A$-critical then the CS will be in $\mathbf{CS}_A$.
If $(n;d,e)$ is not $A$-critical because one of the players is rich then the CS
will not be in $\mathbf{CS}_A$.
If $(n;d,e)$ is not $A$-critical because one of the players is poor and the other one is not rich
then the CS will be in $\mathbf{CS}_A$.

\begin{definition}
A {\it solution set for $A$} is a subset $X \subset \mathbf{CS}_A$ such that:

\begin{itemize}

\item For all $(i, b, \bdag) \in X$,  one of the following holds, where we let 
$(i';b',{\bdag}') = (i-a_1 \bmod m;\bdag - \cII(i, a_1),b - \cI(i, a_1))$:

\begin{itemize}
\item $(i'; b', {\bdag}') \in \mathbf{CS}_A \backslash X$;
\item $b' \geq 0$ and ${\bdag}'< 0$;
\item $b' < 0$ and ${\bdag}' < 0$ and $W_A(i') = \II$.
\end{itemize}
\item For all $(i, b, \bdag) \in \mathbf{CS}_A \backslash X$, and for all $a \in A$, one of the following 
holds, where we let $(i'; b', {\bdag}') = (i-a \bmod m; \bdag - \cII(i, a), b - \cI(i, a))$:

\begin{itemize}
\item $(i';b',{\bdag}') \in X$;
\item $b' < 0$ and ${\bdag}' \geq 0$;
\item $b' < 0$ and ${\bdag}' < 0$ and $W_A(i') = \I$.
\end{itemize}
\end{itemize}
\end{definition}

The following lemma has a straightforward proof that we leave to the reader.

\begin{lemma}\label{le:notpoor}~
\begin{enumerate}
\item
If $d > \gI(n)$ then $d-a_1 > \gII(n-a_1)$.
\item
Let $a\in A$.
If $e>\gII(n)$ then $e>\gI(n-a)$.
\end{enumerate}
\end{lemma}

\begin{theorem}\label{th:main}
Suppose $A$ is cash-periodic with solution set $X$. Suppose $(n;d,e)$ is $A$-critical. 
Then $\WAcash(n;d,e) = \I$ iff $(i_n;b_{n,d,e}, b^{\dagger}_{n,d,e}) \in X$.
\end{theorem}

\begin{proof}
We prove this by induction on $n$.

\noindent
{\bf Base Case 1:} 
$0\le n\le a_1-1$. Since $(n;d,e)$ is $A$-critical $d<\fI(n)=0$.
Hence this case can never occur.

\noindent
{\bf Induction Hypothesis:} Assume the theorem holds for all $n'<n$ and
that $n\ge a_1$.

\noindent
{\bf Induction Step:} Let $(n;d,e)$ be $A$-critical. 
Let $(i;b_{n;d,e},\bdag_{n;d,e})=(i;b,\bdag)$ be the corresponding
state.

\noindent
{\bf Case I:} $(i;b,\bdag) \in X$. 
We show that if Player \I removes $a_1$ then he wins.
Let $(i';b',{\bdag}')$ be the CS that happens when
Player \I removes $a_1$. We want to prove that any NwCS that maps to this CS is
a state where Player \II wins. 
By the definition of $X$ one of the following occurs.

\begin{enumerate}
\item
$(i'; b', {\bdag}') \in \mathbf{CS}_A \backslash X$ and 
the NwCS is $A$-critical. Since 
the real number of stones is some $n'<n$, by the induction hypothesis
Player \II wins.
\item
$(i'; b', {\bdag}') \in \mathbf{CS}_A \backslash X$ and 
the NwCS is not $A$-critical. 
Since $(i'; b', {\bdag}') \in \mathbf{CS}_A \backslash X$ neither player is rich.
By Lemma~\ref{le:notpoor} neither player is poor.
Hence this case cannot occur.
\item 
$b' \geq 0$ and ${\bdag}'< 0$. Player \I is not rich and Player \II is rich,
so Player \II wins by Theorem~\ref{th:rich}.
\item $b' < 0$ and ${\bdag}' < 0$ and $W_A(i') = \II$.
Both Players are rich but $W_A(i')=\II$, so Player \II wins by
Theorem~\ref{th:rich}.
\end{enumerate}

\bigskip

\noindent
{\bf Case 2:}
We show that if Player \I removes any $a\in A$ then he loses.
Let $(i';b',{\bdag}')$ be the corresponding state that happens when
Player \I removes $a$. We want to prove that from this corresponding
state Player \I wins. 
By the definition of $X$ one of the following occurs.
\begin{enumerate}
\item
$(i'; b', {\bdag}') \in X$ and the NwCS is $A$-critical.
Since the real number of stones is some $n'<n$, by the induction hypothesis
Player \I wins.
\item
$(i'; b', {\bdag}') \in X$ and the NwCS is not $A$-critical.
Since $(i'; b', {\bdag}') \in \mathbf{CS}_A \backslash X$ neither player is rich.
By Lemma~\ref{le:notpoor}.2 Player \I is not poor in the new state. Hence
Player \II in the new state is poor so Player \I wins.
\item 
$b'<0$ and ${\bdag}'\ge  0$. Player \I is rich and Player \II is not rich,
so Player \I wins by Theorem~\ref{th:rich}.
\item $b' < 0$ and ${\bdag}' < 0$ and $W_A(i') = \I$.
Both Players are rich but $W_A(i')=\I$, so Player \I wins by
Theorem~\ref{th:rich}.
\end{enumerate}
\end{proof}

\section{ $A = \{1, L\}$ for $L$ even.}\label{se:1L}

We consider NIM games where $A = \{1, L\}$ for some $L$. Since if $L$ is odd, this is basically just $A = \{1\}$, we consider only the case where $L$ is even.

So fix $L$ even for the rest of the section, say $L = 2 \ell$.

\begin{lemma}
$W(n) = \II$ iff $n \equiv 0, 2, 4, \ldots, L-2 \bmod{L+1}$.
\end{lemma}

\begin{lemma}
Let $n \in \mathbf{N}$. Write $n = k(L+1) + i$, where $0 \leq i < L+1$.
\begin{enumerate}
\item
If $i < L$ then
$f_A((L+1)k+i)=Lk+\ceil{\frac{i}{2}}$.
\item
$f_A((L+1)k+L)=L(k+1)$.
\end{enumerate}
\end{lemma}

\begin{lemma}
Let $n \in \mathbf{N}$. Write $n = k(L+1) + i$ where $0 \leq i < L+1$.
\begin{enumerate}
\item If $n < L$ then $f_A^\bot(n) = \floor{\frac{n}{2}}$.
\item If $n \geq L$ and $i < L$ then $f_A^\bot(n) = Lk + \floor{\frac{i}{2}} - \ell + 1$.
\item If $i = L$ then $f_A^\bot(n) = Lk + \ell$.
\end{enumerate}
\end{lemma}

\begin{lemma}
Let $n \in \mathbf{N}$.
\begin{enumerate}
\item If $n$ is even then $\gI(n) = \frac{n}{2} + 1$ and $\gII(n) = \frac{n}{2}$.
\item If $n$ is odd then $\gI(n) = \frac{n-1}{2} + 1$ and $\gII(n) = \frac{n-1}{2} + 1$.
\end{enumerate}
\end{lemma}
\begin{lemma}
$A$ is cash periodic, with period $L+1$. Moreover:

\begin{enumerate}
\item \begin{itemize}
\item For all $i < L$, $\cI(i, 1) = 0$.
\item $\cI(L,1) = L-1$.
\item For all $i$, $\cII(i, 1) = 0$.
\end{itemize}
\item \begin{itemize}
\item For all $i$,  $\cI(i, L) = 0$.
\item For all $i \not= L-1$, $\cII(i, L) = L-1$.
\item $\cII(L-1, L) = 0$.
\end{itemize}
\end{enumerate}

\end{lemma}
\begin{proof}
Easy to check, given all of the preceding lemmas.
\end{proof}

\begin{lemma}
Define $X \subset \mathsf{CS}_A$ as follows:
\begin{itemize}
\item If $i \in \{0, 2, \ldots, L-2\}$ then 
$(i, x, y) \in X$ iff $x < \floor{\frac{y}{L-1}}(L-1)$.
\item If $i \in \{1, 3, \ldots, L-1, L\}$ then $(i, x, y) \in X$ iff $y \geq \floor{\frac{x}{L-1}}(L-1)$.
\end{itemize}

Then $X$ is a solution set.
\end{lemma}

\begin{theorem}\label{th:1L}
Let $n,d,e\in \nat$.
Let $\fI$, $\fII$, $\gI$, and $\gII$ be as defined in this section. (Recall that $\fI$ and $\fII$ are defined simply from $f_A$ and $f_A^\bot$.
\begin{enumerate}
\item
If $d\ge \fI(n)$ and $e < \fII(n)$ then $\WAcash(n;d,e)=\I$.
\item
If $d < \fI(n)$ and $e\ge \fII(n)$ then $\WAcash(n;d,e)=\II$.
\item
If $d\ge \gI(n)$ and $e \ge \gII(n)$ then $\WAcash(n;d,e)=W_A(n)$.
\item
If $d\ge \gI(n)$ and $e < \gII(n)$ then $\WAcash(n;d,e)=\I$.
\item If $d < \gI(n)$ and $e \geq \gII(n)$ then $\WAcash(n; d, e, ) = \II$.
\item If $d < \gI(n)$ and $e < \gI(n)$ then $\WAcash(n; d, e) = \I$ iff $\floor{\frac{d}{a_1}} = \floor{\frac{e}{a_1}}$.
\item
If $\gI(n) \leq d < \fI(n)$ and $\gII(n) \leq e < \fII(n)$, then let 
$n\equiv i \pmod {L+1}$,
$d=\fI(n)-b-1$, and
$e=\fII(n)-\bdag$.
$\WAcash(n;d,e) = \I$ iff $(i, b, \bdag) \in X$.
\end{enumerate}
\end{theorem}

\section{$A = \{1, L, L+1\}$ for $L$ odd}\label{se:1LLodd}

We consider NIM games where $A = \{1, L, L+1\}$ for some odd $L$, say $L = 2\ell + 1$.

Note that $\gI$ and $\gII$ are the same as before since these functions only depend on $\min(A)$.

\begin{lemma}
$W_A(n)=\II$ iff $n\equiv 0,2,4,\ldots,L-1 \bmod {2L+1}$.
\end{lemma}

\begin{lemma}
For all $k$:
\begin{itemize}
\item
$f_A((2L+1)k+i)=\frac{(3L+1)k}{2}+\ceil{\frac{i}{2}}$ for $0\le i < L+1$.
\item
$f_A((2L+1)k+i)=\frac{(3L+1)k}{2}+L+\ceil{\frac{i-L}{2}}$ for $L+1\leq i < 2L +1$.
\end{itemize}
\end{lemma}

\begin{lemma}
For all $k$:
\begin{itemize}
\item
$f_A^\bot((2L+1)k+i)=\frac{(3L+1)k}{2}+\floor{\frac{i}{2}}$ for $0\le i < L+2$.
\item
$f_A^\bot((2L+1)k+i)=\frac{(3L+1)k}{2}+L + \floor{\frac{i-L}{2}}$ for $L+2 \le i < 2L+1$.
\end{itemize}
\end{lemma}

\begin{lemma}
$A$ is cash-periodic with period $2L+1$. Moreover:
\begin{enumerate}
\item \begin{itemize}
\item For all $i < 2L+1$, with $i \not= L+1, L+2$, $\cI(i, 1) = 0$.
\item $\cI(L+1, 1) = \ell+1$.
\item $\cI(L+2, 1) = \ell$.
\item For all $i < 2L+1$, $\cII(i, 1) = 0$.
\end{itemize}

\item \begin{itemize}
\item $\cI(0, L) = 0$.
\item For $0 < i < L+1$, $\cI(i, L) = -\ell$.
\item For $i \geq L+1$ even, $\cI(i, L) = 1$.
\item For $i \geq L+1$ odd, $\cI(i, L) = 0$.
\item For $0 \leq i \leq L+1$, $\cII(i, L) = \ell$.
\item For $i > L+1$ even, $\cII(i, L) = L-1$.
\item For $i > L+1$ odd, $\cII(i, L) = L$.
\end{itemize}

\item \begin{itemize}
\item For all $i \not \in [2, L]$, $\cI(i, L+1) = 0$.
\item For all even $i \in [2, L]$ even, $\cI(i, L+1) = -\ell$.
\item For all odd $i \in [2, L]$, $\cI(i, L+1) = -\ell + 1$.
\item For all $i \not \in [1, L+1]$, $\cII(i, L+1) = L$.
\item For all even $i \in [1, L+1]$, $\cII(i, L+1) = \ell$.
\item For all odd $i \in [1, L+1]$, $\cII(i, L+1) = \ell-1$.
\end{itemize}
\end{enumerate}
\end{lemma}
\begin{proof}
One can see that this is straightforward to check without checking it. Note that this complexity does not show up in the final theorem; it instead reflects the number of cases necessary to consider in its proof. 
\end{proof}

\begin{lemma}
Define $X \subset \mathsf{CS}_A$ as follows.
\begin{itemize}
\item If $i < L+1$ is even or $i > L+1$ is odd then $(i,x , y) \in X$ iff $x \leq \floor{\frac{y}{L-1}}(L-1)$.
\item If $i < L+1$ is odd or $i > L+1$ is even then $(i, x, y) \in X$ iff $y > \floor{\frac{x}{L-1}}(L-1)$.
\item If $i = L+1$ then $(i, x, y) \in X$ iff $y \geq \floor{\frac{x}{L-1}}(L-1)$.
\end{itemize}
Then $X$ is a solution set for $A$.
\end{lemma}

This yields an explicit description of $\WAcash$, as for $A = \{1, L\}$.

\section{$A = \{1, L, L+1\}$ for $L$ Even}\label{se:1LLeven}

We consider NIM games where $A = \{1, L, L+1\}$ for some even $L$. So fix $L$ even for the rest of this section, say $L = 2 \ell$.

Note that $\gI$ and $\gII$ are the same as before since these functions only depend on $\min(A)$.

\begin{lemma}
$W_A(n)=\II$ iff $n\equiv 0,2,4\ldots,L-2 \bmod {2L}$.
\end{lemma}

\begin{lemma}
For all $k$:
\begin{itemize}
\item $f_A(2Lk+i)=\frac{3Lk}{2}+\ceil{\frac{i}{2}}$ for $0\le i < L$.
\item $f_A(2Lk+i)=\frac{3Lk}{2}+L+\ceil{\frac{i-L}{2}}$ for $L\le i\le 2L-1$.
\end{itemize}
\end{lemma}

\begin{lemma}
For all $k$:
\begin{itemize}
\item $f^\bot_A(2Lk+i)=\frac{3Lk}{2}+\floor{\frac{i}{2}}$ for $0\le i < L+1$.
\item $f^\bot_A(2Lk+i)=\frac{3Lk}{2}+L+\floor{\frac{i-L}{2}}$ for $L+1 \leq i < 2L$.
\end{itemize}
\end{lemma}

\begin{lemma}
$A$ is strictly cash-periodic with period $2L$. Moreover, if we let $\mathbf{G}_A = (G, \lI, \lII, h)$ then:

\begin{enumerate}
\item \begin{itemize}
\item For all $i < 2L$ with $i \not= L, L+1$, $\cI(i, 1) = 0$.
\item For $i = L, L+1$, $\cI(i, 1) = \ell$.
\item For all $i < 2L$, $\cII(i, 1) = 0$.
\end{itemize}

\item \begin{itemize}
\item For all even $i \not \in [1, L-1]$, $\cI(i, L) = 0$.
\item For all odd $i \not \in [1, L-1]$, $\cI(i, L) = 1$.
\item For all even $i \in [1, L-1]$, $\cI(i, L) = -\ell$.
\item For all odd $i \in [1, L-1]$, $\cI(i, L) = -\ell + 1$.
\item For all even $i < L+1$, $\cII(i, L) = \ell$.
\item For all odd $i < L+1$, $\cII(i, L) = \ell-1$.
\item For all even $i \geq L+1$, $\cII(i, L) = L$.
\item For all odd $i \geq L+1$, $\cII(i, L) = L-1$.
\end{itemize}

\item \begin{itemize}
\item For all $i \not \in [2, L-1]$, $\cI(i, L+1) = 0$.
\item For all $i \in [2, L-1]$, $\cI(i, L+1) = -\ell$.
\item For all $i \not \in [1, L]$, $\cII(i, L+1) = L$.
\item For all $i \in [1, L]$, $\cII(i, L+1) = \ell$.
\end{itemize}
\end{enumerate}
\end{lemma}

\begin{lemma}
Define $X \subset \mathsf{CS}_A$ as follows:

\begin{itemize}
\item If $i < 2L$ is odd, or if $i = L$, then $(i, x, y) \in X$ iff $y \geq \floor{\frac{b}{\ell}} \ell$.
\item If $i < 2L$ is even and $i \not= L$, then $(i, x, y) \in X$ iff $x < \floor{\frac{y}{\ell}}\ell$.
\end{itemize}
Then $X$ is a solution set for $A$.
\end{lemma}

\section{Conjecture about $A=\{L,\ldots,M\}$}\label{se:LM}

We have written a program that will, for a set $A$, produce a candidate for $\fI,\fII$
and $X$. Note that $\gI,\gII$ only depend on $a_1=\min(A)$ so it is trivial to obtain $\gI$.
Based on this programs output we have the following conjectures about $A=\{L,\ldots,M\}$.

\begin{conjecture}
Let $L\le M$ and $A=\{L,\ldots,M\}$. Then
\begin{enumerate}
\item
There is an offset $\Theta$, depending on $L, M$ such that
\begin{enumerate}
\item
$(\forall n\ge \Theta)[\fI(n+L+M) = \fI(n) + M]$
\item
$(\forall n\ge \Theta)[\fII(n+L+M) = \fII(n) + M]$
\end{enumerate}
\item
$\Theta \leq 5(M-L)^2 +2$ 
\item
If $M \geq 2L$ then $\Theta = 2(L+1)$.
\end{enumerate}
\end{conjecture}

\begin{definition}
Let $X$ be the set of all triples $(i, b, b^{\dagger})$ where:

\begin{itemize}
\item $0 \leq i < L+M$ and $b, b^{\dagger} \geq 0$.
\item if $i < L$ then $\floor{\frac{b}{L}} \leq \floor{\frac{b^{\dagger}}{L}}$.
\item If $L \leq i < 2L$ then $\floor{\frac{b}{L}} \leq \floor{\frac{b^{\dagger}-L}{L}}$.
\item If $2L \leq i < 3L$ then $\floor{\frac{b}{L}} \leq \floor{\frac{b^{\dagger}- 3L + i + 1}{L}}$.
\item If $i \geq 3L$ then $\floor{\frac{b}{L}} \leq \floor{\frac{b^{\dagger}}{L}}$.
\end{itemize}
\end{definition}

\noindent \textbf{Conjecture 2.} Suppose $(n,d,e)$ is $A$-critical. Let $(i, b, b^{\dagger})$ be $(n \bmod (L+M), \fI(n) - d - 1, \fII(n)-e-1)$. Then $WC_A(n,d,e) = \I$ iff $(i, b, b^{\dagger}) \in X$.

\vspace{2 mm}

Hence $WC_A$ is poly-log in $(n,d,e)$ and quadratic in $L,M$.

\section{Summary and Open Questions}\label{se:questions}

We have proven general theorems about who wins $\NIM(A)$ with cash
when either (1) at least one of the players is rich, or
(2) at least one of the players is poor. We have also determined
some conditions so that we can determine what happens when both
players are middle class. We applied these theorems to determine
exactly who wins when $A=\{1,L\}$ and $A=\{1,L,L+1\}$. We also have
a conjecture for $A=\{L,\ldots,M\}$.

For every finite set $A$ there is a nice form for the functions $\fI$ and $\fII$.
Using what we know about $A=\{1,L\}$ and $A=\{1,L,L+1\}$, and our conjecture about
$\{L,\ldots,M\}$
the following conjecture seems reasonable:

{\it For all finite sets $A$ there exists $A,B,c_0,\ldots,c_{L-1}$
such that for all $m\in\nat$, for all $0\le i\le L-1$,
$\fI(Am+i)=Bm+c_i$. Similar for $\fII$.} 

Alas this is not true. 
Let $A=\{3,5,6,10,11\}$. 
The functions $\fI$ and $\fII$ are in the appendix.
They violate the conjecture in two ways: (1) the values of $\fI$ and $\fII$ for
$n\le 63$ do not follow a nice pattern, and (2) the values of $\fI$ and $\fII$ for
$n\ge 64$ have a pattern mod 16 (yeah!) but the value of $B$ in the conjecture
is sometimes 10 and sometimes 11, so not just one value.
In light of the counterexample here is a conjecture:

{\it For all finite sets $A$ there exists $M,A,B_0,\ldots,B_{L-1},c_0,\ldots,c_{L-1}$
such that for all $m\ge M$, for all $0\le i\le L-1$,
$\fI(Am+i)=B_im+c_i$. Similar for $\fII$. }

More generally, is there always a nice win condition? We think so and
state two conjectures about this.

\begin{itemize}
\item
There is an algorithm that will, given a finite set $A$, output
a win condition for $\NIM(A)$.
\item
For every finite set $A$ there is a win condition for $\NIM(A)$.
\end{itemize}

\section{Acknowledgments}

The authors gratefully acknowledge the financial support of
the Maryland Center for Undergraduate Research.
The authors would also like to Steve Cable and 
Sam Zbarsky for helpful discussions.

\vfill\eject

\section{Appendix: $\fI$ and $\fII$ for $A=\{3,5,6,10,11\}$}

For $0\le n\le 63$ $\fI$ and $\fII$ do not follow any real pattern.
For $k\ge 4$ the following holds.

\[
\begin{array}{rl}
\fI(16k+0) & =   11k+3\cr
\fI(16k+1) & = 10k+3\cr
\fI(16k+2) & = 11k+5\cr
\fI(16k+3) & = 10k+5\cr
\fI(16k+4) & = 10k+3\cr
\fI(16k+5) & = 11k+3\cr
\fI(16k+6) & = 10k+5\cr
\fI(16k+7) & = 10k+6\cr
\fI(16k+8) & = 11k+6\cr
\fI(16k+9) & = 10k+8\cr
\fI(16k+10) & = 11k+10\cr
\fI(16k+11) & = 10k+10\cr
\fI(16k+12) & = 10k+8\cr
\fI(16k+13) & = 11k+11\cr
\fI(16k+14) & = 10k+10\cr
\fI(16k+15) & = 10k+11\cr
\end{array}
\]

\[
\begin{array}{rl}
\fII(16k+0) & =   11k\cr
\fII(16k+1) & = 10k\cr
\fII(16k+2) & = 11k\cr
\fII(16k+3) & = 11k+3\cr
\fII(16k+4) & = 11k-1\cr
\fII(16k+5) & = 11k+5\cr
\fII(16k+6) & = 11k+3\cr
\fII(16k+7) & = 11k+5\cr
\fII(16k+8) & = 11k+5\cr
\fII(16k+9) & = 10k+5\cr
\fII(16k+10) & = 11k+3\cr
\fII(16k+11) & = 11k+6\cr
\fII(16k+12) & = 11k+5\cr
\fII(16k+13) & = 11k+10\cr
\fII(16k+14) & = 11k+6\cr
\fII(16k+15) & = 11k+10\cr
\end{array}
\]


\end{document}